\documentclass[reqno]{amsart}
\usepackage{amsmath,amssymb,amsthm}
\numberwithin{equation}{section}

\newtheorem{Theorem}{Theorem}[section]
\newtheorem{Definition}[Theorem]{Definition}
\newtheorem{Corollary}[Theorem]{Corollary}
\newtheorem{Lemma}[Theorem]{Lemma}
\newtheorem{Proposition}[Theorem]{Proposition}
\newtheorem{Remark}{Remark}[section]
\begin{document}

%[[[ Information for Document
\title[global non-existence of semirelativistic equations]%
{A note for the global non-existence of
semirelativistic equations with non-gauge invariant power type nonlinearity}

%[[[ K. Fujiwara
\author[K. Fujiwara]{Kazumasa Fujiwara}

\address{%
Centro di Ricerca Matematica Ennio De Giorgi\\
Scuola Normale Superiore\\
Pisa, Italy.
}

\email{kazumasa.fujiwara@sns.it}
%]]]

%[[[ abstract
\begin{abstract}
The non-existence of global solutions
for semirelativistic equations with non-gauge invariant power type nonlinearity
is revisited by a relatively direct way
with a pointwise estimate of fractional derivative of some test functions.
\end{abstract}
%]]]

\maketitle
%]]]

%[[[ \section{Introduction}
\section{Introduction}
%[[[ We consider the Cauchy problem the for following
We consider the Cauchy problem for the following
semirelativistic equations with non-gauge invariant power type nonlinearity:
	\begin{align}
	\begin{cases}
	i \partial_t u + (-\Delta)^{1/2} u = \lambda |u|^p,
	&\quad t \in \lbrack 0, T), \quad x \in \mathbb R^n,\\
	u(0) = u_0,
	&\quad x \in \mathbb R^n,
	\end{cases}
	\label{eq:1.1}
	\end{align}
with $\lambda \in \mathbb C \backslash \{0\}$,
where $\partial_t = \partial/\partial t$
and $\Delta$ is the Laplacian in $\mathbb R^n$.
Here $(- \Delta)^{1/2}$ is realized as a Fourier multiplier with symbol
$|\xi|$: $(-\Delta)^{1/2} = \mathfrak F^{-1} |\xi| \mathfrak F$,
where $\mathfrak F$ is the Fourier transform defined by
	\[
	(\mathfrak F u )(\xi) = \hat u(\xi)
	= (2 \pi)^{-n/2} \int_{\mathbb R^n} u(x) e^{-i x \cdot \xi} dx.
	\]
%]]]

%[[[ back ground
We remark that the Cauchy problem such as \eqref{eq:1.1} arises
in various physical settings and accordingly,
semirelativistic equations are also called
half-wave equations, fractional Schr\"odinger equations, and so on,
see \cite{bib:2,bib:10,bib:11} and reference therein.
%]]]

%[[[ The local existence for \eqref{eq:1.1} in the $H^s(\mathbb R^n)$ framework
The local existence for \eqref{eq:1.1} in the $H^s(\mathbb R^n)$ framework
is easily seen if $s > n / 2$,
where $H^s(\mathbb R^n)$ is the usual Sobolev space
defined by $(1-\Delta)^{-s} L^2(\mathbb R^n)$.
Here the local existence in the $H^s(\mathbb R^n)$ framework means
that for any $H^s(\mathbb R^n)$ data,
there is a positive time $T$ such that
there is a solution for the corresponding integral equation,
	\begin{align}
	u(t)
	= e^{it(-\Delta)}u_0 - i \lambda \int_0^t e^{i(t-t')(-\Delta)^{1/2}} |u(t')|^p dt',
	\label{eq:1.2}
	\end{align}
in $C([0,T);H^s(\mathbb R^n))$.
We remark that for $s > n/2$, local solution for \eqref{eq:1.2}
may be constructed by a standard contraction argument
with the Sobolev embedding $H^s(\mathbb R^n) \hookrightarrow L^\infty(\mathbb R^n)$
which holds if and only if $s > n/2$.
We also remark that in the one dimensional case,
$s > 1/2$ is also the necessary condition for the local existence
in the $H^s(\mathbb R)$ framework
because the non-existence of local weak solutions for \eqref{eq:1.1}
with some $H^{1/2}(\mathbb R)$ data is shown in \cite{bib:6}.
In a general setting, the necessary condition is still open
and partial results are discussed in \cite{bib:10}.
On the other hand,
\eqref{eq:1.1} is scaling invariant.
Namely, when $u$ is a solution for \eqref{eq:1.1} with initial data $u_0$,
then for any $\rho > 0$, the pair,
	\begin{align}
	u_\rho(t,x) = \rho^{1/(p-1)} u(\rho t, \rho x),
	\quad u_{0,\rho} = \rho^{1/(p-1)} u_0(\rho x)
	\label{eq:1.3}
	\end{align}
also satisfies \eqref{eq:1.1}.
Then the case where $(s,q)$ satisfies that for $u_0 \in H_q^s \backslash \{ 0 \}$,
	\[
	\| (-\Delta)^{s/2} u_{0,\rho} \|_{L^2(\mathbb R^n)}
	\to \infty \quad \mathrm{as} \quad \rho \to \infty
	\ \Longleftrightarrow \ 
	s - \frac{n}{q} + \frac{1}{p-1}
	> 0
	\]
is called $H^s_q(\mathbb R^n)$ scaling subcritical case,
where $H_q^s(\mathbb R^n) = (1-\Delta)^{-s/2} L^q(\mathbb R^n)$.
Moreover, if $ s = \frac{n}{q} + \frac{1}{p-1}$,
we call the case as $H^s_q(\mathbb R^n)$ scaling critical case.
In the $H^s(\mathbb R^n)$ scaling subcritical case, in general,
the local existence in $H^s(\mathbb R^n)$ framework
is expected but this is not our case.
Similarly,
in the $H_q^s(\mathbb R^n)$ scaling subcritical and critical cases,
the non-existence of global solution for some $H_q^s(\mathbb R^n)$ data is expected.
In the present paper,
we are interested in a priori global non-existence in some scaling subcritical case.
%]]]

%[[[ In this note, we revisit the global non-existence of \eqref{eq:1.1}.
In the present paper, we revisit the global non-existence of \eqref{eq:1.1}.
In order to go back to prior works,
we define weak solutions for \eqref{eq:1.1} and its lifespan.
\begin{Definition}
Let $u_0 \in L^2(\mathbb R^n)$.
We say that $u$ is a weak solution to \eqref{eq:1.1} on $[0,T)$,
if $u$ belongs to
$L_\mathrm{loc}^1(0, T ; L^2(\mathbb R^n))
\cap L_\mathrm{loc}^1(0, T ; L^p(\mathbb R^n))$
and the following identity
	\begin{align}
	\int_0^\infty \big( u(t)
	\big| i \partial_t \psi (t) + (- \Delta)^{1/2} \psi (t) \big) dt
	= i (u_0 | \psi (0)) + \lambda \int_0^\infty \big( |u(t)|^p \big|\psi  (t) \big) dt
	\label{eq:1.4}
	\end{align}
holds for any
$\psi \in C([0,\infty); H^1(\mathbb R^n)) \cap C^1([0,\infty); L^2(\mathbb R^n))$
satisfying
	\[
	\mathrm{supp}\thinspace \psi \subset [0,T] \times \mathbb R,
	\]
where $(\cdot \mid \cdot)$ is the usual $L^2(\mathbb R^n)$ inner product defined by
	\[
	(f \mid g) = \int_{\mathbb R^n} \overline{f(x)} g(x) dx.
	\]
Moreover we define $T_w$ as
	\[
	T_w
	= \inf\{T > 0 \ ; \ \mbox{There is no weak solutions for \eqref{eq:1.1} on $[0,T)$.}\}.
	\]
\end{Definition}
\noindent
%]]]

%[[[ The author and Ozawa \cite{bib:7} showed the global non-existence
The author and Ozawa \cite{bib:7} showed the global non-existence
in $L^1(\mathbb R)$ scaling critical and subcritical cases.
%[[[ Proposition : Fujiwara Ozawa
\begin{Proposition}
\label{Proposition:1.2}
If $n=1$, $1 < p \leq 2$, and $u_0 \in (L^1 \cap L^2) (\mathbb R)$ satisfying that
	\begin{align}
	\mathrm{Re} (\overline \lambda u_0) =0,
	\quad
	- \mathrm{Im} \bigg( \int_{\mathbb R} \overline \lambda u_0(x) dx \bigg) > 0,
	\label{eq:1.5}
	\end{align}
then there is no global weak solution,
namely, if $T$ is big enough,
there is no weak solution on $[0,T)$.
\end{Proposition}
\noindent
Here we remark that the case when $p=2$ is $L^1(\mathbb R)$ scaling critical.
%]]]
%]]]

%[[[ Later, Inui \cite{bib:10} generalized the non-existence result as follows:
Later, Inui \cite{bib:10} obtained the following global non-existence
in $H^s(\mathbb R^n)$ scaling critical and subcritical cases
for large data with $0 \leq s < n/2$
and in $L^2(\mathbb R^n)$ scaling subcritical case for small data:
%[[[ Proposition :
\begin{Proposition}
\label{Proposition:1.3}
Let $ s\geq 0$.
We assume that $1 < p \leq 1 + 2/(n - 2s)$
and the initial value $u_0 (x) = \mu f (x)$,
where $\mu > 0$ and $f \in L^2(\mathbb R^n)$ satisfies that
	\begin{align}
	\mathrm{Re}(\overline \lambda f ) = 0,
	\quad
	-\mathrm{Im}(\overline \lambda f ) \geq
	\begin{cases}
	|x|^{-k},
	&\quad \mathrm{if} \quad|x| \leq 1,\\
	0,
	&\quad \mathrm{if} \quad|x| > 1,
	\end{cases}
	\label{eq:1.6}
	\end{align}
with $k < n/2 - s (\leq 1/(p - 1))$.
Then there exists $\mu_0$ such that if $\mu > \mu_0$,
then there is no global weak solution.
Moreover, for any $\mu \in \lbrack \mu_0 , \infty)$,
there exists a positive constant $C > 0$ such that
	\[
	T_w \leq C \mu^{- \frac{1}{\frac{1}{p-1}-k}}.
	\]
\end{Proposition}
%]]]

%[[[ Proposition :
\begin{Proposition}
\label{Proposition:1.4}
We assume that $1 < p < 1 + 2/n$,
and the initial value $u_0 (x) = \mu f (x)$,
where $\mu > 0$ and $f \in L^2(\mathbb R^n)$ satisfies that
	\begin{align}
	\mathrm{Re}(\overline \lambda f ) = 0,
	\quad
	-\mathrm{Im}(\overline \lambda f ) \geq
	\begin{cases}
	0,
	&\quad \mathrm{if} \quad|x| \leq 1,\\
	|x|^{-k},
	&\quad \mathrm{if} \quad|x| > 1,
	\end{cases}
	\label{eq:1.7}
	\end{align}
where $n/2 < k < 1/(p - 1)$.
Then there is no global weak solution.
Moreover, there exist $\epsilon > 0$ and a positive constant $C > 0$ such that
	\begin{align}
	T_w \leq
	\begin{cases}
	C \mu^{-\frac{1}{\frac{1}{p-1}-k}},
	&\quad \mathrm{if} \quad 0< \mu < \epsilon ,\\
	2,
	&\quad \mathrm{if} \quad \mu > \epsilon.
	\end{cases}
	\end{align}
\end{Proposition}
%]]]

%[[[ We remark that he also showed
\noindent
We remark that for $0 < s < n/2$,
there is some $H^s(\mathbb R^n)$ function satisfying \eqref{eq:1.6}.
For details, see \cite[Example 5.1]{bib:9}.
Moreover, he showed Proposition \ref{Proposition:1.3} with mass term.
For details, see \cite[Theorem 1.2]{bib:10}.
%]]]
%]]]

%[[[ The aim of this note is to revisit Propositions
The aim of the present paper is to give an alternative relatively direct proof
of  Propositions \ref{Proposition:1.3} and \ref{Proposition:1.4}.
In \cite{bib:7,bib:10},
the non-existence of weak solutions are shown by a test function method
introduced by Baras-Pierre \cite{bib:1} and Zhang \cite{bib:12,bib:13}.
However, standard test function method is not applicable to \eqref{eq:1.1}
because the method relies on pointwise control of derivative of test functions.
Namely, the classical Leibniz rule plays a critical role.
On the other hand, since $(-\Delta)^{1/2}$ is non-local,
$\mathrm{supp} \thinspace (-\Delta)^{1/2} \phi$
is bigger than $\mathrm{supp} \thinspace \phi$
for $\phi \in C_0^{\infty}(\mathbb R^n)$ in general.
Therefore,
it is impossible to have the following pointwise estimate:
There exists a positive constant $C$ such that
for any $\phi \in C_0^{\infty}(\mathbb R^n)$,
	\begin{align}
	|((-\Delta)^{1/2} \phi^\ell)(x)|
	\leq C |\phi^{\ell-1}(x) ((-\Delta)^{1/2} \phi)(x)|,
	\quad \forall x \in \mathbb R^n
	\label{eq:1.9}
	\end{align}
with $\ell > 1$.
In order to avoid from the difficulty of fractional derivative,
in \cite{bib:7,bib:10},
\eqref{eq:1.1} is transformed into
	\begin{align}
	\partial_t^2 v - \Delta v
	= - |\lambda |^2 \partial_t |u|^p,
	\label{eq:1.10}
	\end{align}
where $v = \mathrm{Im}(\overline \lambda u)$.
\eqref{eq:1.10} may be obtained by
applying $- \mathrm{Im} ( \overline \lambda (i \partial t - (-\Delta)^{1/2}))$
to \eqref{eq:1.1}.
In the present paper,
on the other hand,
we are interested in showing global non-existence without using \eqref{eq:1.10}
but by introducing the following estimate:
%]]]

%[[[ Lemma : Fractional Derivative
\begin{Lemma}
\label{Lemma:1.5}
Let $\langle x \rangle = ( 1 + |x|^2)^{1/2}$.
For $q > 0$,
there exists a positive constant $A_{n,q}$ depending only on $n$ and $q$
such that for any $x \in \mathbb R^n$,
	\begin{align}
	|( (-\Delta)^{1/2} \langle \cdot \rangle^{-q} ) (x) |
	\leq
	\begin{cases}
	A_{n,q} \langle x \rangle^{-q-1},
	&\quad \mathrm{if} \quad 0 < q < n,\\
	A_{n,q}\langle x \rangle^{-n-1} (1+\log (1+|x|)),
	&\quad \mathrm{if} \quad q = n,\\
	A_{n,q} \langle x \rangle^{-n-1},
	&\quad \mathrm{if} \quad q > n.
	\end{cases}
	\label{eq:1.11}
	\end{align}
\end{Lemma}
%]]]

%[[[ Lemma \ref{Lemma:1.5} may be shown by a direct computation
Lemma \ref{Lemma:1.5} may be shown by a direct computation
with the following representation:
	\begin{align}
	((-\Delta)^{1/2} f)(x)
	&= B_{n,s} \thinspace \mathrm{P.V.} \int_{\mathbb R^n} \frac{f(x) - f(x+y)}{|y|^{n+1}} dy
	\nonumber\\
	&= B_{n,s} \thinspace \lim_{\epsilon \searrow 0}
	\int_{|y| \geq \epsilon} \frac{f(x) - f(x+y)}{|y|^{n+1}} dy,
	\label{eq:1.12}
	\end{align}
where
	\[
	B_{n,s}
	= \bigg( \int_{\mathbb R^n} \frac{1-\cos(\xi_1)}{|\xi|^{n+1}} d \xi \bigg)^{-1}.
	\]
For details of this representation,
for example, we refer the reader \cite{bib:5}.
If one regards $(-\Delta)^{1/2}$ as $\nabla$,
Lemma \ref{Lemma:1.5} seems natural at least for $0 < q < n$.
When $q \geq n $, the decay rate of fractional derivative is worse
than the expectation form the classical first derivative
but it is sufficient for our aim and actually sharp.
The dominating term of the fractional derivative for $q \geq n$
appears when $|x|/2 \leq |y| \leq 2 |x|$ hold.
Indeed,
	\[
	|x|/2 \leq |y| \leq 2 |x|
	\Longleftarrow
	|x+y| \leq |x|/2
	\]
and therefore
	\[
	\int_{|x|/2 \leq |y| \leq 2|x|} (1+|x+y|)^{-q} |y|^{-n-1} dy
	\geq 2^{-n-1} |x|^{-n-1} \int_{|z| \leq |x|/2} (1+|z|)^{-q} dz.
	\]
For details, see the proof of Lemma \ref{Lemma:1.5}
and Remarks \ref{Remark:2.1} and \ref{Remark:2.2} in Section 2.
%]]]

%[[[ We also remark that C\'ordoba and C\'ordoba \cite{bib:3} showed that
We also remark that C\'ordoba and C\'ordoba \cite{bib:3} showed that
	\begin{align}
	(-\Delta)^{s/2} (\phi^2 )(x) \leq 2 \phi(x) ((-\Delta)^{s/2} \phi)(x)
	\label{eq:1.13}
	\end{align}
for any $0 \leq s \leq 2$, $\phi \in \mathcal S(\mathbb R^2)$, and $x \in \mathbb R^2$,
where $\mathcal S$ denotes the collection of rapidly decreasing functions.
In general, $\phi \geq 0$ does not imply $(-\Delta)^{s/2} \phi \geq 0$,
and therefore \eqref{eq:1.13} does not imply \eqref{eq:1.9}.
D'Abbicco and Reissig \cite{bib:4} studied global non-existence for
structural damped wave equation possessing fractional derivative
by generalizing \eqref{eq:1.13}.
For the study of structural damped wave equation,
\eqref{eq:1.13} works well because we have non-negative solutions(\cite[Lemma 1]{bib:4}),
which we cannot expect for \eqref{eq:1.1}.
%]]]

%[[[ Lemma \ref{Lemma:1.5}
Lemma \ref{Lemma:1.5} implies the following statement,
which is our main statement:
%]]]

%[[[ Proposition : Nonexistence,
\begin{Proposition}
\label{Proposition:1.6}
Let
	\[
	X(T)
	= C([0,T);L^2(\mathbb R^n))
	\cap C^1([0,T),H^{-1}(\mathbb R^n))
	\cap L^\infty(0,T;L^{p}(\mathbb R^n)).
	\]
Let $u_0 \in L^2(\mathbb R^n)$ satisfy that
	\begin{align}
	M_R(0)
	> C_{n,p,\alpha} R^{n-1/(p-1)},
	\label{eq:1.14}
	\end{align}
with some $R >0$ and $\alpha \in \mathbb C$ satisfying that
	\begin{align}
	\mathrm{Re} ( \alpha \lambda) > 0,
	\label{eq:1.15}
	\end{align}
where $M_R(0)$ and $C_{n,p,\alpha}$ is given by
	\begin{align*}
	M_R(0)
	&= - \mathrm{Im}
	\bigg( \alpha \int_{\mathbb R^n} u_0(x) \langle x / R \rangle^{-n-1} dx \bigg),\\
	C_{n,p,\alpha}^p
	&=
	2^{1+p'/p} p^{-p'/p} p'^{-1} \mathrm{Re}(\alpha \lambda)^{-p'} | \alpha|^{p+p'}
	A_{n,n+1}^{p'}
	\bigg(\int_{\mathbb R^n} \langle x \rangle^{-n-1} dx \bigg)^p.
	\end{align*}
Then there is no solution for \eqref{eq:1.1} in $X(T)$
with $u(0) = u_0$ and $T > T_{n,p,\lambda,\alpha,R}$, where
	\begin{align*}
	T_{n,p,\lambda,\alpha,R}
	&= (p-1)^{-1} D_{n,p,\lambda,\alpha}^{-1}
	R^{n(p-1)}( M_R(0) - C_{n,p,\alpha} R^{n-1/(p-1)} )^{-p+1},\\
	%%%%%%%%%
	D_{n,p,\lambda,\alpha}
	&= 2^{-1} \mathrm{Re}(\alpha \lambda) |\alpha|^{-p}
	\bigg( \int_{\mathbb R^n} \langle x \rangle^{-n-1} dx \bigg)^{-p+1}.
	\end{align*}
\end{Proposition}
%]]]

%[[[ Proposition \ref{Proposition:1.6} follows from Lemma \ref{Lemma:1.5}
Proposition \ref{Proposition:1.6} follows from Lemma \ref{Lemma:1.5}
by using an ordinary differential equation (ODE) approach
introduced by the author and Ozawa \cite{bib:8}.
Indeed, it is shown that for some $R > 0$,
	\[
	F(t)
	= - \mathrm{Im}
	\bigg( \alpha \int_{\mathbb R^n} u(t,x) \langle x / R \rangle^{-n-1} dx \bigg)
	- C_{n,p,\alpha} R^{n-1/(p-1)}
	\quad
	(F(0) > 0),
	\]
is a super solution of an ODE taking the form of $f' = f^p$,
coming from \eqref{eq:1.1} without $(-\Delta)^{1/2} u$.
Therefore $F$ cannot exists globally
and is shown to blow up at $t = T_{n,p,\lambda,\alpha,R}$.
We remark that $L^2(\mathbb R^n)$ solution may blow up or
loose its sense before the blowup time of $F$.
We also remark that this approach is considered relatively direct
comparing to test function method
because with test function method,
since solutions are canceled out in weak equations,
it is impossible to see the behavior of blowup solutions.
On the other hand, in our approach,
a rough a priori behavior of weighted integral of solutions
is obtained.
%]]]

%[[[ Proposition \ref{Proposition:1.3} and Proposition \ref{Proposition:1.4}
Proposition \ref{Proposition:1.6} is our main statement because
in scaling subcritical case,
Propositions \ref{Proposition:1.2}, \ref{Proposition:1.3} and \ref{Proposition:1.4}
may be obtained as corollaries of Proposition \ref{Proposition:1.6}.
%]]]

%[[[ Corollary : Fujiwara Ozawa
\begin{Corollary}
\label{Corollary:1.7}
Let $1 < p < 1 + 1/n$.
Let $\alpha \in \mathbb C$ and $u_0 \in (L^1 \cap L^2)(\mathbb R^n)$
satisfy \eqref{eq:1.15} and
	\begin{align}
	-\mathrm{Im} \bigg( \alpha \int_{\mathbb R^n} u_0 (x) dx \bigg) > 0.
	\label{eq:1.16}
	\end{align}
Then there exists no solution in $X(T)$ for sufficiently large $T$.
\end{Corollary}
%]]]

%[[[ Corollary : Inui 1
\begin{Corollary}
\label{Corollary:1.8}
Let $u_0 (x) = \mu f (x)$ where $\mu \gg 1$ and $f$ satisfies
	\begin{align}
	-\mathrm{Im}(\alpha f (x) ) \geq
	\begin{cases}
	|x|^{-k},
	&\quad \mathrm{if} \quad|x| \leq 1,\\
	0,
	&\quad \mathrm{if} \quad|x| > 1,
	\end{cases}
	\label{eq:1.17}
	\end{align}
with some $k < \min(n/2,1/(p-1))$ and $\alpha$ satisfying \eqref{eq:1.15}.
Then there exists some $R_1 > 0$ satisfying \eqref{eq:1.14}
and
	\[
	T_{n,p,\lambda,\alpha,R_1}
	\leq C \mu ^{-\frac{1}{1/(p-1)-k}}.
	\]
\end{Corollary}
%]]]

%[[[ Corollary : Inui 2
\begin{Corollary}
\label{Corollary:1.9}
Let $u_0 (x) = \mu f (x)$ where $0 < \mu \ll 1$ and $f$ satisfies
	\begin{align}
	-\mathrm{Im}(\alpha f(x) ) \geq
	\begin{cases}
	0,
	&\quad \mathrm{if} \quad|x| \leq 1,\\
	|x|^{-k},
	&\quad \mathrm{if} \quad|x| > 1,
	\end{cases}
	\label{eq:1.18}
	\end{align}
with some $n/2 < k < 1/(p-1)$ and $\alpha$ satisfying \eqref{eq:1.15}.
Then there exists some $R_2 > 0$ satisfying \eqref{eq:1.14}
and
	\begin{align}
	T_{n,p,\lambda,\alpha,R_2}
	\leq C \mu ^{-\frac{1}{1/(p-1)-\min(n,k)}}.
	\label{eq:1.19}
	\end{align}
\end{Corollary}
%]]]

%[[[ \eqref{eq:1.16} with \eqref{eq:1.15} is more general than \eqref{eq:1.5}.
Since $\alpha$ can be $\overline \lambda$,
\eqref{eq:1.16}, \eqref{eq:1.17} and \eqref{eq:1.18} with \eqref{eq:1.15}
are more general conditions than \eqref{eq:1.5}, \eqref{eq:1.6}, and \eqref{eq:1.7}.
Moreover, when $k > n$, \eqref{eq:1.19} is sharp
from the view point of the scaling transformation \eqref{eq:1.3}
as long as one tries to estimate the lifespan
with $L^1(\mathbb R^n)$ norm of initial data.
Indeed, the lifespan depends on $\rho^{-1}$
and
	\[
	\| u_{0,\rho} \|_{L^1(\mathbb R^n)}^{-\frac{1}{\frac{1}{p-1}-n}}
	= \| u_{0} \|_{L^1(\mathbb R^n)}^{-\frac{1}{\frac{1}{p-1}-n}} \rho^{-1}.
	\]
%]]]

%[[[ In the sections bellow,
In the sections bellow,
we give the proof of each statements.
%]]]
%]]]

%[[[ \section{Proof of Lemma \ref{Lemma:1.5}}
\section{Proof of Lemma \ref{Lemma:1.5}}

%[[[ Proof part
%[[[ We estimate $(-\Delta)^{1/2} \langle \cdot \rangle^{-q}$ with \eqref{eq:1.12}
We estimate $(-\Delta)^{1/2} \langle \cdot \rangle^{-q}$ with \eqref{eq:1.12}
pointwisely without $B_{n,q}$.
%]]]

%[[[ At first,
At first,
we recall
	\begin{align}
	|x+y|^2 -|x|^2
	= 2 x \cdot y + |y|^2 = y \cdot (2 x + y).
	\label{eq:2.1}
	\end{align}
and the fact that for any $ r > 0$,
	\begin{align}
	\mathrm{P.V.} \int_{|y| < r} y |y|^{-n-1} dy = 0.
	\label{eq:2.2}
	\end{align}

%]]]

%[[[ For $|x| \leq 1$,
For $|x| \leq 1$,
we divide integral domain into the following two parts:
	\begin{align*}
	\Omega_1 &= \{(x,y) : \ |y| \geq 1 \},\\
	\Omega_2 &= \{(x,y) : \ |y| \leq 1 \}.
	\end{align*}
For $\Omega_1$, it is easy to see that
	\[
	\sup_{|x| \leq 1} \bigg| \int_{\Omega_1}
	\frac{\langle x+y \rangle^{-q} - \langle x \rangle^{-q}}{|y|^{n+1}} dy \bigg|
	< \infty.
	\]
For $\Omega_2$, we rewrite $\langle x+y \rangle^{-q}$,
by applying the Taylor theorem for $(1+\cdot)^{-q/2}$, as
	\begin{align}
	\langle x+y \rangle^{-q}
	&= \langle x \rangle^{-q}
	- \frac{q}{2} \langle x \rangle^{-q-2} (|x+y|^2 - |x|^2)
	+ \frac{q(q+2)}{2^2} R(x,y),
	\label{eq:2.3}\\
	R(x,y)
	&= \int_{|x|^2}^{|x+y|^2}
	(1+\rho)^{-q/2-2} (|x+y|^2 - \rho) d \rho.
	\nonumber
	\end{align}
Then the principle value for first approximation of $\langle x+y \rangle^{-q}$
may be computed with \eqref{eq:2.1} and \eqref{eq:2.2} by
	\begin{align*}
	&\mathrm{P.V.}
	\int_{\Omega_2} \langle x \rangle^{-q-2} (|x+y|^2 - |x|^2) |y|^{-n-1} dy\\
	&= 2 \langle x \rangle^{-q-2} x \cdot
	\mathrm{P.V.} \int_{|y| < 1} y |y|^{-n-1} dy
	+ \langle x \rangle^{-q-2} \int_{|y| < 1} |y|^{-n+1} dy\\
	&= \omega_n \langle x \rangle^{-q-2}
	\leq \omega_n,
	\end{align*}
where $\omega_n$ is the volume of the unit sphere $S^{n-1}$.
Moreover,
since
	\begin{align*}
	| R(x,y) |
	&\leq (1+\min(|x|^2,|x+y|^2))^{-q/2-2} ||x+y|^2 - |x|^2|^2\\
	&\leq |y|^2 |y+2x|^2
	\leq 9 |y|^2,
	\end{align*}
the principle value for the remainder is estimated by
	\[
	\bigg| \mathrm{P.V.} \int_{\Omega_2} R(x,y) |y|^{-n-1} dy \bigg|
	\leq 9 \int_{|y| \leq 1} |y|^{-n+1} dy
	\leq 9 \omega_n.
	\]
%]]]

%[[[ For $|x| \geq 1$,
For $|x| \geq 1$,
we divide integral domain into the following three parts:
	\begin{align*}
	\Omega_3 &= \{(x,y) : \ |y| \geq 2|x| \},\\
	\Omega_4 &= \{(x,y) : \ \frac{1}{2} |x| \leq |y| \leq 2|x| \},\\
	\Omega_5 &= \{(x,y) : \ |y| \leq \frac{1}{2} |x| \}.
	\end{align*}
For $\Omega_3$, since $|x+y| \geq |x|$,
	\[
	\bigg| \int_{\Omega_3}
	\frac{\langle x+y \rangle^{-q} - \langle x \rangle^{-q}}{|y|^{n+1}} dy
	\bigg|
	\leq 2 \omega_n \langle x\rangle^{-q} \int_{2|x|}^\infty r^{-2} dr
	= \omega_n \langle x\rangle^{-q} |x|^{-1}.
	\]
For $\Omega_4$, since $|y| \sim |x|$,
	\begin{align*}
	\int_{\Omega_4} \langle x \rangle^{-q} | y |^{-n-1} dy
	&\leq 2^{n+1} \langle x \rangle^{-q} |x|^{-n-1} \int_{|y| \leq 2|x|} dy\\
	&\leq 2^{2n+1} n^{-1} \omega_n \langle x \rangle^{-q} |x|^{-1}
	\end{align*}
and
	\begin{align*}
	&\int_{\Omega_4} \langle x + y \rangle^{-q} | y |^{-n-1} dy\\
	&\leq
	2^{q/2+n+1} | x |^{-n-1}
	\int_{|x+y| \leq 3|x|} (1+|x+y|)^{-q} dy
	\\
	&\leq 2^{q/2+n+1} \omega_n
	\begin{cases}
	(n-q)^{-1} | x |^{-n-1} (1+3|x|)^{n-q},
	&\quad \mathrm{if} \quad 0 \leq q < n,\\
	| x |^{-n-1} \log (1+3|x|),
	&\quad \mathrm{if} \quad q = n,\\
	(q-n)^{-1} | x |^{-n-1},
	&\quad \mathrm{if} \quad q > n,\\
	\end{cases}
	\end{align*}
where we have used the fact that,
	\[
	(a+b)^{1/2} \geq 2^{-1/2} (a^{1/2} + b^{1/2}).
	\]
For $\Omega_5$, we again use the expansion \eqref{eq:2.3}.
The principle value for first approximation is computed by
	\begin{align*}
	&\mathrm{P.V.}
	\int_{\Omega_5} \langle x \rangle^{-q-2} (|x+y|^2 - |x|^2) |y|^{-n-1} dy\\
	&= \langle x \rangle^{-q-2} \int_{|y| < |x|/2} |y|^{-n+1} dy
	= 2^{-1} \omega_n \langle x \rangle^{-q-2} |x|.
	\end{align*}
Moreover, the remainder is estimated by
	\begin{align*}
	| R(x,y) |
	&\leq (1+\min(|x|^2,|x+y|^2))^{-q/2-2} ||x+y|^2 - |x|^2|^2\\
	&\leq (1+|x|^2/4)^{-q/2-2} |y|^2 |y+2x|^2\\
	&\leq 2^{q+2} \cdot 5^2 \langle x \rangle^{-q-4} |x|^2 |y|^2.
	\end{align*}
Therefore the principle value for the remainder is estimated by
	\begin{align*}
	&\bigg| \mathrm{P.V.} \int_{\Omega_5} R(x,y) |y|^{-n-1} dy \bigg|\\
	&\leq 2^{q+2} \cdot 5^2 \langle x \rangle^{-q-4} |x|^2
	\int_{|y| \leq |x|/2} |y|^{-n+1} dy\\
	&= 2^{q+1} \omega_n \cdot 5^2 \langle x \rangle^{-q-4} |x|^3.
	\end{align*}
%]]]
%]]]

%[[[ Remark : Optimality
\begin{Remark}
\label{Remark:2.1}
When $q \geq n$,
there exists some positive constant $E_q$ and $R$ such that
for any $|x| \geq R$,
	\[
	((- \Delta)^{1/2} \langle \cdot \rangle) (x)
	\leq
	\begin{cases}
	- E_n \langle x \rangle^{-n-1} \log (1+|x|/2),
	&\quad \mathrm{if} \quad q = n.\\
	- E_q \langle x \rangle^{-n-1},
	&\quad \mathrm{if} \quad q > n,
	\end{cases}
	\]
Indeed, in the previous proof, on $\Omega_4^c$,
the size of principle value is estimated by $\langle x \rangle^{-q-1}$
with some constants.
On the other hand, on $\Omega_4$,
\begin{align*}
	&\int_{\Omega_4} (\langle x \rangle^{-q} - \langle x + y \rangle^{-q})
	| y |^{-n-1} dy\\
	&\leq
	 2^{n+1} \langle x \rangle^{-q} | x |^{-n-1}
	\int_{|x+y| \leq 3|x|} dy
	-2^{n+1} | x |^{-n-1}
	\int_{|x+y| \leq |x|/2} (1+|x+y|)^{-q} dy
	\\
	&\leq
	n^{-1} 2^{n+1} 3^n \omega_n | x |^{-1}\langle x \rangle^{-q}
	- 2^{2n} \omega_n | x |^{-n-1} \int_1^{|x|/2} (1+r)^{-q+n-1} dr
\end{align*}
since
	\[
	\{ y ; |x+y| \leq |x|/2 \}
	\subset \Omega_4
	\subset \{ y ; |x+y| \leq 3 |x| \}.
	\]
\end{Remark}
%]]]

%[[[ Remark : More functions
\begin{Remark}
\label{Remark:2.2}
A similar phenomena may happen with some $L^1(\mathbb R^n)$ functions
decaying quicker than $\langle x \rangle^{-n}$.
Indeed, by a similar computation show that
there exists some positive constant $C$ such that for $|x| \gg 1$,
	\[
	((-\Delta)^{1/2} e^{-|\cdot |^2})(x)
	\leq - C \langle x \rangle^{-n-1}.
	\]
\end{Remark}
%]]]
%]]]

%[[[ \section{Proof of Proposition \ref{Proposition:1.6}}
\section{Proof of Proposition \ref{Proposition:1.6}}
Assume that there exists a solution $u$ for \eqref{eq:1.1}
belonging to $X(T)$ with $ T > T_{n,p,\lambda,\alpha,R}$.
Let
	\[
	M_R(t) = - \mathrm{Im} \bigg( \alpha
	\int_{\mathbb R^n} u_0(x) \langle x/R \rangle^{-n-1} dx \bigg).
	\]
Then
	\begin{align}
	\frac{d}{dt} M_R(t)
	&= \mathrm{Re} \bigg( \alpha
	\int_{\mathbb R^n} i \partial_t u(t,x) \langle x/R \rangle^{-n-1}dx \bigg)
	\nonumber\\
	&= \mathrm{Re} (\alpha \lambda)
	\int_{\mathbb R^n} |u(t,x)|^p \langle x/R \rangle^{-n-1}dx
	\nonumber\\
	&- R^{-1} \mathrm{Re} \bigg( \alpha
	\int_{\mathbb R^n} u(t,x) ((-\Delta)^{1/2} \langle \cdot \rangle^{-n-1})(x/R)dx \bigg)
	\nonumber\\
	&\geq \mathrm{Re} (\alpha \lambda)
	\int_{\mathbb R^n} |u(t,x)|^p \langle x/R \rangle^{-n-1} dx
	\nonumber\\
	&- A_{n,n+1} R^{-1} |\alpha|
	\int_{\mathbb R^n} |u(t,x)| \langle x/R \rangle^{-n-1} dx.
	\label{eq:3.1}
	\end{align}
By the H\"older and Young inequalities,
	\begin{align}
	&A_{n,n+1} |\alpha| R^{-1} \int_{\mathbb R^n} |u(t,x)| \langle x/R \rangle^{-n-1} dx
	\nonumber\\
	&\leq A_{n,n+1} |\alpha| R^{-1}
	\bigg( \int_{\mathbb R^n} \langle x/R \rangle^{-n-1} dx\bigg)^{1/p'}
	\bigg( \int_{\mathbb R^n} |u(t,x)|^p \langle x/R \rangle^{-n-1} dx\bigg)^{1/p}
	\nonumber\\
	&\leq A_{n,n+1} |\alpha| R^{n/p'-1}
	\bigg( \int_{\mathbb R^n} \langle x \rangle^{-n-1} dx\bigg)^{1/p'}
	\bigg( \int_{\mathbb R^n} |u(t,x)|^p \langle x/R \rangle^{-n-1} dx\bigg)^{1/p}
	\nonumber\\
	&\leq p^{-p'/p} p'^{-1} 2^{p'/p} \mathrm{Re}(\alpha \lambda)^{-p'/p} | \alpha|^{p'}
	A_{n,n+1}^{p'}
	R^{n-p'} \int_{\mathbb R^n} \langle x \rangle^{-n-1} dx
	\nonumber\\
	&+ 2^{-1} \mathrm{Re}(\alpha \lambda)
	\int_{\mathbb R^n} |u(t,x)|^p \langle x/R \rangle^{-n-1} dx
	\label{eq:3.2}
	\end{align}
and
	\begin{align}
	|M_R(t)|
	&= \bigg| \mathrm{Im}
	\bigg( \alpha \int_{\mathbb R^n} u(t,x) \langle x/R \rangle^{-n-1} dx \bigg) \bigg|
	\nonumber\\
	&\leq
	|\alpha| R^{n/p'}
	\bigg( \int_{\mathbb R^n} \langle x \rangle^{-n-1} dx \bigg)^{1/p'}
	\bigg( \int_{\mathbb R^n} |u(t,x)|^p \langle x/R \rangle^{-n-1} dx \bigg)^{1/p}.
	\label{eq:3.3}
	\end{align}
Combining \eqref{eq:3.1}, \eqref{eq:3.2}, and \eqref{eq:3.3},
	\begin{align*}
	&\frac{d}{dt} M_R(t)\\
	&\geq 2^{-1} \mathrm{Re}(\alpha \lambda)
	\int_{\mathbb R^n} |u(t,x)|^p \langle x/R \rangle^{-n-1} dx\\
	&- p^{-p'/p} p'^{-1} 2^{p'/p} \mathrm{Re}(\alpha \lambda)^{-p'/p} | \alpha|^{p'}
	A_{n,n+1}^{p'}
	R^{n-p'} \int_{\mathbb R^n} \langle x \rangle^{-n-1} dx\\
	&\geq 2^{-1} \mathrm{Re}(\alpha \lambda) |\alpha|^{-p}
	\bigg( \int_{\mathbb R^n} \langle x \rangle^{-n-1} dx \bigg)^{-p+1}
	R^{-n(p-1)} M_R(t)^p\\
	&- p^{-p'/p} p'^{-1} 2^{p'/p} \mathrm{Re}(\alpha \lambda)^{-p'+1} | \alpha|^{p'}
	A_{n,n+1}^{p'}
	R^{n-p'} \int_{\mathbb R^n} \langle x \rangle^{-n-1} dx\\
	&\geq 2^{-1} \mathrm{Re}(\alpha \lambda) |\alpha|^{-p} R^{-n(p-1)}
	\bigg( \int_{\mathbb R^n} \langle x \rangle^{-n-1} dx \bigg)^{-p+1}
	( M_R(t)^p - C_{n,p,\alpha}^p R^{np-p'} )\\
	&\geq D_{n,p,\lambda,\alpha} R^{-n(p-1)}
	(M_R(t) - C_{n,p,\alpha} R^{n-1/(p-1)})^p.
	\end{align*}
Therefore,
	\begin{align}
	&M_R(t) - C_{n,p,\alpha} R^{n-1/(p-1)}
	\nonumber\\
	&\geq \{( M_R(0) - C_{n,p,\alpha} R^{n-1/(p-1)} )^{-p+1}
	- (p-1) D_{n,p,\lambda,\alpha} R^{-n(p-1)} t \}^{-1/(p-1)}
	\label{eq:3.4}\\
	&> 0,
	\nonumber
	\end{align}
and the RHS of \eqref{eq:3.4} blows up at $ t = T_{n,p,\lambda,\alpha,R}$
and so does $M_R$.
Since
	\[
	M_R(t)
	\leq
	\| u(t) \|_{L^2(\mathbb R^n)}
	\| \langle \cdot / R \rangle^{-n-1} \|_{L^2(\mathbb R^n)},
	\]
\eqref{eq:3.4} contradicts the existence of solutions in $C(0,T;L^2(\mathbb R^n))$
with $ T > T_{n,p,\lambda,\alpha,R}$.
%]]]

%[[[ \section{Proof of Corollaries}
\section{Proof of Corollaries}
\label{section:4}
%[[[ \begin{proof}[Proof of Corollary \ref{Corollary:1.7}]
\begin{proof}[Proof of Corollary \ref{Corollary:1.7}]
By the Lebesgue convergence theorem,
	\[
	\lim_{R \to \infty} M_R(0)
	= - \mathrm{Im} \bigg( \alpha \int_{\mathbb R^n} u_0(x) dx \bigg)
	> 0
	\]
and $R^{n-1/(p-1)} \to 0$ as $ R \to \infty$.
Therefore, \eqref{eq:1.14} is satisfied with some $R > 0$.
\end{proof}
%]]]

%[[[ \begin{proof}[Proof of Corollary \ref{Corollary:1.8}]
\begin{proof}[Proof of Corollary \ref{Corollary:1.8}]
For $0 < R < 1$,
	\begin{align*}
	M_R(0)
	&\geq \mu \int_{|x| \leq 1} |x|^{-k} \langle x/R \rangle^{-n-1} dx\\
	&\geq 2^{-n-1} \mu \int_{|x| \leq R} |x|^{-k} dx\\
	&= (n-k)^{-1} 2^{-n-1} \omega_n \mu R^{n-k}.
	\end{align*}
Let $I_1 = (n-k)^{-1} 2^{-n-1} \omega_n$ and
	\[
	R_1 = (\mu I_1/(2C_{n,p,\alpha}))^{-\frac{1}{1/(p-1)-k}},
	\]
where $R_1 <1$ if $\mu \gg 1$.
Then
	\begin{align*}
	M_{R_1}(0) - C_{n,p,\alpha} R_1^{n-1/(p-1)}
	&\geq R_1^{n-k} ( \mu I_1 - C_{n,p,\alpha} R_1^{k-1/(p-1)})\\
	&\geq R_1^{n-k} \mu I_1/2 > 0
	\end{align*}
and therefore \eqref{eq:1.14} is satisfied.
Moreover,
	\begin{align*}
	T_{n,p,\lambda,\alpha,R_1}
	&\leq (p-1)^{-1} D_{n,p,\lambda,\alpha}^{-1}
	(\mu I_1/(2C_{n,p,\alpha}))^{\frac{k(p-1)}{k-1/(p-1)}} (\mu I_1)^{-p+1}\\
	&= (p-1)^{-1} D_{n,p,\lambda,\alpha}^{-1}
	(2C_{n,p,\alpha})^{-\frac{k(p-1)}{k-1/(p-1)}}
	(\mu I_1)^{-\frac{1}{1/(p-1)-k}}.
	\end{align*}
\end{proof}
%]]]

%[[[ \begin{proof}[Proof of Corollary \ref{Corollary:1.9}]
\begin{proof}[Proof of Corollary \ref{Corollary:1.9}]
For $R \gg 1$,
	\begin{align*}
	M_R(0)
	&\geq \mu \int_{|x| \geq 1} |x|^{-k} \langle x/R \rangle^{-n-1} dx\\
	&\geq 2^{-n-1} \mu \int_{1 \leq |x| \leq R} |x|^{-k} dx\\
	&\geq 2^{-n-1} \omega_n \mu \int_1^R r^{n-k-1} dr,\\
	&\geq 2^{-n-1} \omega_n \mu
	\begin{cases}
	(n-k)^{-1} (R^{n-k} - 1),
	&\quad \mathrm{if} \quad k < n,\\
	\int_1^{2} r^{n-k-1} dr,
	&\quad \mathrm{if} \quad k \geq n,
	\end{cases}\\
	&\geq I_2 \mu R^{(n-k)_+},
	\end{align*}
where $(n-k)_+ = \max(n-k,0)$ and
	\[
	I_2
	= \begin{cases}
	2^{-n-2} \omega_n (n-k)^{-1},
	&\quad \mathrm{if} \quad k < n,\\
	2^{-n-1} \omega_n \int_1^{2} r^{n-k-1} dr,
	&\quad \mathrm{if} \quad k \geq n.
	\end{cases}
	\]
Let
	\[
	R_2 = (\mu I_2/(2C_{n,p,\alpha}))^{-\frac{1}{1/(p-1)-\min(n,k)}},
	\]
where $R_2 \gg 1$ if $\mu \ll 1$.
Then
	\begin{align*}
	&T_{n,p,\lambda,\alpha,R_2}\\
	&\leq (p-1)^{-1} D_{n,p,\lambda,\alpha}^{-1}
	R_2^{n(p-1)-(n-k)_+(p-1)}
	( \mu I_2 - C_{n,p,\alpha} R_2^{\min(n,k) - 1/(p-1)})^{-p+1}\\
	&= (p-1)^{-1} D_{n,p,\lambda,\alpha}^{-1}
	R_2^{\min(n,k)(p-1)}
	( \mu I_2 - C_{n,p,\alpha} R_2^{\min(n,k) - 1/(p-1)})^{-p+1}\\
	&\leq (p-1)^{-1} D_{n,p,\lambda,\alpha}^{-1}
	(2C_{n,p,\alpha})^{-\frac{\min(n,k)(p-1)}{\min(n,k)-1/(p-1)}}
	(\mu I_2)^{-\frac{1}{1/(p-1)-\min(n,k)}}.
	\end{align*}
\end{proof}
%]]]
%]]]

%[[[ \section*{Acknowledgment}
\section*{Acknowledgment}
The author is grateful to professor Vladimir Georgiev
for his helpful comment on the optimality of \eqref{eq:1.11}
for $q =n$.
%]]]

%[[[thebibliograpy made by biblist
%[[[ \providecommand{\bysame}{\leavevmode\hbox to3em{\hrulefill}\thinspace}
\providecommand{\bysame}{\leavevmode\hbox to3em{\hrulefill}\thinspace}
\providecommand{\MR}{\relax\ifhmode\unskip\space\fi MR }
\providecommand{\MRhref}[2]{%
  \href{http://www.ams.org/mathscinet-getitem?mr=#1}{#2}
}
\providecommand{\href}[2]{#2}

%]]]
\end{document}